\documentclass{amsart}

\makeatletter
 \def\LaTeX{\leavevmode L\raise.42ex
   \hbox{\kern-.3em\size{\sf@size}{0pt}\selectfont A}\kern-.15em\TeX}
\makeatother

\newcommand{\BibTeX}{{\rm B\kern-.05em{\sc
i\kern-.025emb}\kern-.08em\TeX}}
\newtheorem{col}{Corollary}[section]
\newtheorem{thm}{Theorem}[section]
\newtheorem{lem}[thm]{Lemma}

\theoremstyle{defn}
\newtheorem{defn}{Definition}

\numberwithin{equation}{section}

\begin{document}

\title[ Weak Weyl's Law on  compact Riemannian manifolds]{Shannon sampling and Weak Weyl's Law on  compact Riemannian manifolds}

\author{Isaac Z. Pesenson}
\address{Department of Mathematics, Temple University,
Philadelphia, PA 19122} \email{pesenson@temple.edu}

\begin{abstract} 
The well known  Weyl's asymptotic formula gives an approximation to the number $\mathcal{N}_{\omega}$  of eigenvalues (counted with multiplicities) on an interval $[0,\>\omega]$ of the Laplace-Beltrami operator  on a compact Riemannian manifold ${\bf M}$.
In this paper we approach this question from the point of view of Shannon-type sampling on compact Riemannian manifolds.
Namely, we give a direct proof that   $\mathcal{N}_{\omega}$ is comparable to cardinality of certain sampling
sets for the subspace of $\omega$-bandlimited functions on ${\bf M}$.

\end{abstract}

\maketitle

\section{Introduction}

\subsection{Objectives}

Spectral geometry  concerned with questions which relate spectral properties of operators acting in function spaces on a Riemannian manifold and the geometry of the underlying manifold.
One of the most famous results of such kind is the Weyl's asymptotic formula for the number of eigenvalues of an elliptic (pseudo-)differential operator on a compact Riemannian manifold. 
The goal of this paper  is  to demonstrate that in the case of a general compact Riemannian manifold the so-called weak Weyl's formula closely relates to cardinality of certain  sampling sets for bandlimited functions. This  fact  was first noticed in \cite{Pes04a}.

	\subsection{Weyl's asymptotic formula on compact Riemannian manifolds} Let  $\mathbf{M}$  be  a compact connected Riemannian manifold without boundary and  $\Delta $ is the Laplace-Beltrami operator. It is given in a local coordinate system  by the formula \cite{Hor}
\begin{equation}\label{Lapl}
\Delta
f=\sum_{m,k}\frac{1}{\sqrt{det(g_{ij})}}\partial_{m}\left(\sqrt{det(g_{ij})}
g^{mk}\partial_{k}f\right)
\end{equation}
where $g_{ij}$ are components of the metric tensor,$\>\>det(g_{ij})$ is the determinant of the matrix $(g_{ij})$, $\>\>g^{mk}$ components of the matrix inverse to $(g_{ij})$. 
 The operator is second-order differential self-adjoint and non-negative  in the space $L_{2}({\bf M})$ constructed with respect to Riemannian measure.
  Domains of the powers
 $\Delta^{s/2}, s\in \mathbf{R},$ coincide with the Sobolev spaces
$H^{s}(\mathbf{M}), s\in \mathbf{R}$. 
Since  $\Delta$ is a second-order differential self-adjoint and non-negative  definite operator  on a compact connected  Riemannian manifold  it has  a discrete spectrum $0=\lambda_{0}<\lambda_{1}\leq \lambda_{2},...$ which goes to infinity  without any accumulation points  and there exists a complete  family  $\{u_{j}\}$  of orthonormal eigenfunctions which form a  basis in $L_{2}(\mathbf{M})$ \cite{Hor}.

We will need the following definitions.

\begin{defn}
The space of $\omega$-bandlimited functions $\mathbf{E}_{\omega}(\Delta)$ is defined as the span of all eigenfunctions of $\Delta$ whose eigenvalues are not greater than $\omega.$ The dimension of the subspace $\mathbf{E}_{\omega}(\Delta)$ will be denoted as $\mathcal{N}_{\omega}$.
\end{defn}

One can easily verify that $f$ belongs to  $\mathbf{E}_{\omega}(\Delta)$ if and only if the following Bernstein type inequality holds 

$$
\|\Delta^{k}f\|_{L_{2}(M)} \leq \omega^{k}\|f\|_{L_{2}(M)} 
$$
for all natural $k$.

	\begin{defn}We say that $M_{\rho}=\{x_{j}\},\>x_{j}\in {\bf M},\>\rho>0,$ is a {\bf metric $\rho$-lattice} if 
	
	\begin{enumerate}
	
	\item  Balls $B(x_{j}, \rho/2)$ are disjoint 
	$$
	B(x_{j},\rho/2)\cap B(x_{i}, \rho/2)=\emptyset,\>\>\>\>\>j\neq i,
	$$ 
	but balls $B(x_{\nu}, \rho)$ form a cover of ${\bf M}$. 
	
	\item  There exists a constant $N_{{\bf M}}$ such that multiplicity of all such covers $\left\{B(x_{j}, \rho)\right\}$ is bounded by $N_{{\bf M}}$.

\end{enumerate}
	\end{defn}
	One can show \cite{Pes00}, \cite{Pes04b} existence of metric lattices for sufficiently small  $\rho>0$. We reprove this fact in Lemma \ref{cover} below.	
Note that   $\mathcal{N}_{\omega}$ is the same as the number of eigenvalues (counting with their multiplicities) which are not greater $\omega$.  
	According to the  Weyl's asymptotic formula \cite{Hor} one has for large $\omega$
\begin{equation}\label{Weyl-1}
\mathcal{N}_{\omega}\sim A \>Vol(\mathbf{\mathbf{M}})\omega^{d/2},
\end{equation}
where $d=dim \>\mathbf{\mathbf{M}}$ and $A$ is a constant which is independent on ${\bf M}$. To reveal meaning of the right-hand side of this formula let's rewrite it in  the following form
\begin{equation}
\mathcal{N}_{\omega}\sim A \>Vol(\mathbf{\mathbf{M}})\omega^{d/2}=A \frac{ Vol({\bf M})}{\left(\omega^{-1/2}\right)^{d}}.
\end{equation}	
Since in the case of a Riemannian manifold ${\bf M}$ of dimension $n$ 
all the balls of the same radius $\rho$ have essentially the same volume $\sim \rho^{d}$  
the last fraction  can be  interpreted as a  number of balls $
	B(x_{\nu},\omega^{-1/2})$ whose centers $\{x_{\nu}\}$ form  a lattice $M_{\omega^{-1/2}}$.

	\bigskip
	
{ \bf  The main goal of our paper is to present a direct proof of  the following Theorem \ref{WWL} (which we call the Weak Weyl's Law) 
 without using the Weyl's asymptotic formula (\ref{Weyl-1}).}

		\begin{thm}(Weak Weyl's Law)  \label{WWL}  In the case of a Riemannian manifold the number
		$
		\mathcal{N}_{\omega}
		$
		 of eigenvalues of $\Delta$ in $[0,\>\omega]$ counting with their multiplicities  is equivalent to a number of points in a metric lattice $M_{\omega^{-1/2}}.$   Namely,  there are   constants $a=a({\bf M})>0$ and 
		$$
		0<\gamma=\gamma({\bf M})<1,
		$$
		 such that for all sufficiently large $\omega$ the following double inequality  holds
		\begin{equation}\label{DI}
		a \>\sup |M_{\omega^{-1/2}}|
		\leq \mathcal{N}_{\omega} \leq \> \inf |M_{\gamma\omega^{-1/2}}|,
		\end{equation}
		where $\sup$  is taken over all  $\omega^{-1/2}$-lattices and $\inf$  is taken over all  $\gamma\omega^{-1/2}$-lattices and $|M_{s}|$ denotes cardinality of  a lattice.
	\end{thm}

\section{Covering Lemma}

We consider a compact  Riemannian manifold  ${\bf M}, \dim {\bf M}=d,$ with metric tensor
$g$. It is known that the Laplace-Beltrami operator $\Delta$ which is defined in (\ref{Lapl}) is a self-adjoint positive definite
operator in the corresponding space $L_{2}({\bf M})$ constructed  from
$g$. Domains of the powers
 $\Delta^{s/2}, s\in \mathbb{R},$ coincide with the Sobolev spaces
$H^{s}({\bf M}), s\in \mathbb{R}$. To choose norms on spaces $H^{s}({\bf M}),$
we consider a finite cover of ${\bf M}$ by balls
$B(y_{\nu},\sigma)$ where $y_{\nu}\in {\bf M}$ is the center  of the ball and
$\sigma$ is its radius. For a partition of unity
${\varphi_{\nu}}$ subordinate to the family
$\{B(y_{\nu},\sigma)\}$ we introduce Sobolev space $H^{s}({\bf M})$ as the
completion of
$C_{0}^{\infty}({\bf M})$ with respect to the norm
\begin{equation}
\|f\|_{H^{s}({\bf M})}=\left(\sum_{\nu}\|\varphi_{\nu}f\|^{2}
_{H^{s}(B(y_{\nu},\sigma))}\right)
^{1/2}.
\end{equation}
The regularity Theorem for the Laplace-Beltrami operator $\Delta$
states that the norm (1.1) is equivalent to the graph norm
$\|f\|+\|\Delta^{s/2}f\|$.

 The volume of a ball $B(x,\rho)$ will be denoted
by $|B(x,\rho)|.$  Let us note that  in the case of a compact Riemannian manifold of dimension $d$
 there exist constants $a_{1}=a_{1}({\bf M}), a_{2}=a_{2}({\bf M})$ such that for a  ball $B(x, \rho)$ of sufficiently small radius $\rho$ and any center $x\in {\bf M}$ one has 
 \begin{equation}\label{balls}
 a_{1}\rho ^{d}\leq |B(x, \rho)|\leq a_{2}\rho^{d},
\end{equation}
where
$$
 |B(x, \rho)|=\int_{B(x, \rho)}dx,\>\>\>d=dim\>{\bf M}.
$$
The inequality (\ref{balls})  implies 
  the next inequality  with the same $a_{1}$ and $a_{2}$:
 \begin{equation}\label{2balls}
 \frac{a_{1}}{a_{2}}|B(x_{2},\rho)|\leq |B(x_{1}, \rho)|\leq \frac{a_{2}}{a_{1}}|B(x_{2},\rho)|, \>\>\>\rho<r,
 \end{equation}
 where $x_{1}, x_{2}$ are any two points in ${\bf M}$ and  $r$ is the injectivity radius of the manifold. 
Since ${\bf M}$ is compact there exists a constant $c=c({\bf M})$ such that for any $0<\sigma<\lambda<r/2$ the following
inequality holds true
\begin{equation}\label{doubling}
|B(x,\lambda)|\leq\left(\lambda/\sigma\right)^{d}c|B(x,\sigma)|.
\end{equation}
In what follows we will use the notation 
$$
N_{{\bf M}}=\frac{12^{d}c a_{2}}{a_{1}}.
$$
The following Covering Lemma plays an important role for the
paper.

\begin{lem}\label{cover}
If ${\bf M}$ satisfy the above assumptions then for any $0<\rho<r/6$
there
 exists a finite set of points $\{x_{i}\}$ such that

1) balls $B(x_{i}, \rho/4)$ are disjoint,

2) balls $B(x_{i}, \rho/2)$ form a cover of ${\bf M}$,

3) multiplicity of the cover by balls $B(x_{i},\rho)$ is not greater
$N_{{\bf M}}.$

\end{lem}

\begin{proof}

Let us choose a family of disjoint balls $B(x_{i},\rho/4)$ such
that there is no ball $B(x,\rho/4), x\in {\bf M},$ which has empty intersections
with all balls from our family. Then the family $B(x_{i},\rho/2)$ is a
cover of
${\bf M}$. Every ball from the family $\{B(x_{i}, \rho)\}$, that has
non-empty intersection with a particular ball $\{B(x_{j}, \rho)\}$ is
contained in the ball $\{B(x_{j}, 3\rho)\}$. Since any two balls from the
family $B(x_{i},\rho/4)$
are disjoint, it gives the following estimate for the index of
multiplicity
$N$ of the cover $B(x_{i},\rho)$:
\begin{equation}
N\leq\frac{\sup_{y\in {\bf M}}|B(y,3\rho)|}{\inf_{x\in {\bf M}}|B(x,\rho/4)|}.
\end{equation}
From here, according to (\ref{doubling}) we obtain

$$N\leq\frac{\sup_{y\in {\bf M}}|B(y,3\rho)|}{\inf_{x\in {\bf M}}|B(x,\rho/4)|}\leq 12^{d}c
\frac{\sup_{y\in {\bf M}}|B(y,\rho/4)|}{\inf_{x\in {\bf M}}|B(x,\rho/4)|}
\leq \frac{12^{d}c a_{2}}{a_{1}} =N_{{\bf M}}.
$$

\end{proof}

\section{Sampling sets for bandlimited functions and the upper estimate on the number of eigenvalues. }

\subsection{Poincare-type inequality on manifolds}

One can prove the following Poincare type inequality  (see \cite{Pes00}, \cite{Pes04b}). We sketch it's proof for completeness.

\begin{thm}
There exists a constant $C=C({\bf M},  k)$ such that if $\rho>0$ is sufficiently small then for all $\rho$ lattices $M_{\rho}=\{x_{j}\}$ and all $f\in H^{k}({\bf M}),\>\>\>k>d/2, \>\>\>d=\dim {\bf M},$
\begin{equation}\label{POINC}
\|f\|_{L_{2}({\bf M})} \leq 
C({\bf M},  k)\left\{\rho^{d/2}\left(\sum_{x_{j}\in M_{\rho}}|f(x_{j})|^{2}\right)^{1/2}+\rho^{k}\|\Delta^{k/2}f\|_{L_{2}(M)} \right\}.
\end{equation}
\end{thm}

\begin{proof}

Let $M_{\rho}=\{x_{i}\}$ be a $\rho$-admissible set
and $\{\varphi_{\nu}\}$ the partition of unity from (1.1). For any $f\in
C^{\infty}(M)$, every fixed $B(x_{i},\rho)$ and every
 $x\in B(x_{i},\rho/2)$
$$
(\varphi_{\nu}f)(x)=(\varphi_{\nu}f)(x_{i})+\sum_{1\leq|\alpha|\leq n-1}
\frac{1}{\alpha !}\partial^{\alpha}(\varphi_{\nu}f)(x_{i})(x-x_{i})
^{\alpha}+
$$
\begin{equation}\label{Taylor}
\sum_{|\alpha|=n}\frac{1}{(n-1)!}\int_{0}^{\tau}t^{n-1}\partial
^{\alpha}(\varphi_{\nu}f)(x_{i}+t\vartheta)\vartheta^{\alpha}dt,
\end{equation}
where $x=(x_{1},...,x_{d}), x_{i}=(x_{1}^{i},...,x_{d}^{i}), \alpha=(
\alpha_{1},...,\alpha_{d}), x-x_{i}=(x_{1}-x_{1}^{i})^{\alpha_{1}}...
(x_{d}-x_{d}^{i})^{\alpha_{d}}, \tau=\|x-x_{i}\|,
\vartheta=(x-x_{i})/ \tau.$
By using the Sobolev embedding Theorem one can prove  the following inequality
\begin{equation}
|\partial^{\alpha}(\varphi_{\nu}f)(x_{i})|\leq C_{{\bf M},m}\sum_{|\mu|
\leq m}
\rho^{|\mu+\alpha|-d/2}\|\partial^{\mu+\alpha}(\varphi_{\nu}f)\|
_{L_{2}(B(x_{i},\rho))},
\end{equation}
where $\mu=(\mu_{1}, \mu_{2}, ... ,\mu_{d}), m>d/2.$
It allows the following  estimation of the second term in (\ref{Taylor}).
$$
\int_{B(x_{i},\rho/2)}\left|\sum_{1\leq|\alpha|\leq n-1}
\frac{1}{\alpha !}\partial^{\alpha}(\varphi_{\nu}f)(x_{i})(x-x_{i})
^{\alpha}\right|^{2}dx\leq
$$
$$
C_{{\bf M},n}\sum_{|\gamma|\leq n+m-1}   \rho^{2|\gamma|}
\|\partial^{\gamma}(\varphi_{\nu}f)\|^{2}_{L_{2}(B(x_{i},\rho))}.
$$
Next, to estimate the third term in (\ref{Taylor})
 we use the Schwartz inequality and the assumption $n>d/2$
$$
\left |\int_{0}^{\tau}t^{n-1}\partial
^{\alpha}(\varphi_{\nu}f)(x_{i}+t\vartheta)\vartheta^{\alpha}dt\right|^{2}
\leq
$$
$$
\left(\int_{0}^{\tau}t^{n-d/2-1/2}\left |t^{d/2-1/2}\partial
^{\alpha}(\varphi_{\nu}f)(x_{i}+t\vartheta)\right|dt\right)^{2}\leq
$$
$$
C_{{\bf M},n}\tau^{2n-d}\int_{0}^{\tau}t^{d-1}\left|\partial
^{\alpha}(\varphi_{\nu}f)(x_{i}+t\vartheta)\right|^{2}dt.
$$
We integrate both sides of this inequality over the ball
$B(x_{i},\rho/2)$ using the spherical coordinate system
$(\tau,\vartheta).$
$$
\int_{0}^{\rho/2}\tau^{d-1}\int_{|\vartheta |=1}
\left|\int_{0}^{\tau}t^{n-1}\partial
^{\alpha}(\varphi_{\nu}f)(x_{i}+t\vartheta)\vartheta^{\alpha}dt\right|^{2}d\vartheta
d\tau\leq
$$
$$C_{{\bf M},n}\int_{0}^{\rho/2}t^{d-1}\left(\int_{|\vartheta |=1}\int_{0}^{\rho/2}
\tau^{2n-d}\left|\partial^{\alpha}(\varphi_{\nu}f)(x_{i}+t\vartheta)\right|^{2}
\tau^{d-1}d\tau
d\vartheta\right)dt\leq
$$
$$
C_{{\bf M},n}\rho^{2n}\|\partial^{\alpha}
(\varphi_{\nu}f)\|^{2}_{L_{2}(B(x_{i},\rho))},\>\>\>\tau=\|x-x_{i}\|\leq\rho/2, |\alpha|=n.
$$
Next, for  $n>d/2$ and $k=n+m-1$,
$$\|\varphi_{\nu}f\|^{2}_{L_{2}(B(x_{i},\rho/2))}\leq
C_{1}({\bf M},k)\left(\rho^{d}|f(x_{i}|^{2}+
\sum_{j=1}^{k}\sum_{1\leq|\alpha|\leq j}\rho^{2|\alpha|}
\|\partial^{\alpha}
(\varphi_{\nu}f)\|^{2}_{L_{2}(B(x_{i},\rho))}\right),
$$
where $k>d-1$ since $n>d/2$ and $m>d/2.$
Since balls $B(x_{i},\rho/2)$ cover the manifold and the cover by
$B(x_{i},\rho)$
has a finite multiplicity $\leq N_{{\bf M}}$ the
summation over all balls gives
$$
\|f\|^{2}_{L_{2}({\bf M})}\leq C_{2}({\bf M},k)\left\{\rho^{d}\left(\sum_{i=1}^{\infty}
|f(x_{i})|^{2}\right)+\sum_{j=1}^{k}\rho^{2j}\|f\|^{2}_{H^{j}({\bf M})}
\right\}, k>d-1.
$$
Using this inequality and the regularity theorem for Laplace-Beltrami
operator we obtain
$$
\|f\|_{L_{2}({\bf M})}\leq$$
$$
C_{3}({\bf M},k)\left\{\rho^{d/2}\left(\sum_{i=1}^{\infty}
|f(x_{i})|^{2}\right)^{1/2}+\sum_{j=1}^{k}\rho^{j}\left(\|f\|+
\|\Delta^{j/2}f\|\right)
\right\}, k>d-1.
$$
For the self-adjoint $\Delta$ for any $a>0,\rho>0, 0\leq j\leq k$ we have
the following interpolation inequality
$$
\rho^{j}\|\Delta^{j/2}f\|\leq a^{2k-j}\rho^{2k}\|\Delta^{k}f\|
+c_{k}a^{-j}\|f\|.
$$
Because in the last inequality we are free to choose any $a>0$ we are
coming to
our main claim.
\end{proof}

\subsection{Sampling sets for bandlimited functions and the upper estimate on the number of eigenvalues. }

Now, if a bandlimited function $f $ belongs to $\mathbf{E}_{\omega}(\Delta)$ the  Bernstein inequality implies  
$$
\rho^{k}\|\Delta^{k/2}f\|_{L_{2}({\bf M})}\leq \left(\rho\omega^{1/2}\right)^{k}\|f\|_{L_{2}({\bf M})}.
$$ 
If  $C({\bf M}, k)$ is the same as in (\ref{POINC}) and we pick a such $\rho$ for which
$$
\rho =\gamma \omega^{-1/2},\>\>\>\gamma=\gamma({\bf M})=\frac{(C({\bf M}, k))^{1/k}}{2}<1,
$$
we can move the second term on the right side in (\ref{POINC}) to the left to obtain following Plancherel-Polya-type inequality  which shows that in the spaces of bandlimited functions $\mathbf{E}_{\omega}(\Delta)$ the regular  $L_{2}({\bf M})$ norm is controlled by  a discrete one (in fact, they are equivalent).

\begin{thm}\label{discineq}
There exists a   $0<\gamma=\gamma({\bf M})<1$  and there exists a constant  $C_{1}=C_{1}({\bf M})$
 such that for
any $\omega>0$, every metric $\rho$-lattice $M_{\rho}=\{x_{j}\}$ with
$\rho= \gamma\omega^{-1/2}$ the following inequality  holds true

\begin{equation}
\|f\|_{L_{2}({\bf M})} \leq C_{1}\rho^{d/2}\left(\sum_{x_{j}\in M_{\rho}}
|f(x_{j})|^{2}\right)^{1/2}\label{PP1}
\end{equation}
for all $f\in E_{\omega}(\Delta),\>\>\>d=dim\>\mathbf{M}$. \label{PP}
\end{thm}

	\begin{col} 
	There exists a   $0<\gamma=\gamma({\bf M})<1$  such that  for every $\omega>0$ and every metric $\rho$-lattice $M_{\rho}=\{x_{j}\}$ with
$\rho= \gamma\omega^{-1/2}$ the set $M_{\rho}=\{x_{j}\}$ is a sampling set for the space ${\bf E}_{\omega}(\Delta)$.  
	\end{col}
	In other words,   every function $f\in {\bf E}_{\omega}(\Delta)$ is uniquely determined by its values $\{f(x_{j})\}$ and can be reconstructed from this set of values in a stable way.

	\textit{Since dimension $\mathcal{N}_{\omega}$ of the space $ {\bf E}_{\omega}(\Delta)$ cannot be bigger than cardinality of a sampling set for this space we obtain the following statement}.

	\begin{col}
	There exists a   $0<\gamma=\gamma({\bf M})<1$   such that for any $\omega>0$ 
	\begin{equation}\label{upper}
	\mathcal{N}_{\omega}\leq \>\inf \left|M_{\gamma\omega^{-1/2}}\right|,
	\end{equation}
	where $\left|M_{\gamma\omega^{-1/2}}\right|$ is the number of points in a lattice $M_{\gamma\omega^{-1/2}}$ and $\inf$ is taken over all such lattices.
	\end{col}

\section{The lower estimate}

\subsection{Kernels on compact Riemannian manifolds}

Let $\sqrt{\Delta}$ be the positive square root of a second order differential  elliptic self-adjoint nonnegative operator $\Delta$ in $L_{2}({\bf M})$.
For any measurable  bounded function $F(\lambda),\>\>\lambda\in (-\infty, \infty)$ and any $t>0$ one defines a bounded operator $F(t\sqrt{\Delta})$ by the formula
\begin{equation}\label{func-calc-2}
F(t\sqrt{\Delta})f(x)=\int_{\mathbf{M}}K^{F}_{t}(x,y)f(y)dy=\left<K^{F}_{t}(x,\cdot),f(\cdot)\right>,
\end{equation}
where $f\in L_{2}(\mathbf{M})$ and 
\begin{equation}\label{KERN}
K^{F}_{t}(x,y)=\sum_{l=0}^{\infty} F(t\sqrt{\lambda_l})u_l(x)\overline{u_l(y)} = K^{F}_t(y,x).
\end{equation}
The function $K^{F}_t$ is known as the kernel of  the operator $F(t\sqrt{\Delta})$.

We will need the following lemma.
\begin{lem}\label{order}
If  $0\leq F_{1}\leq F_{2}$ and both of them are bounded and have sufficiently fast decay at infinity then $~K^{F_{1}}_{t}(x,x)\leq K^{F_{2}}_{t}(x,x)~$ for any $x\in {\bf M}$ and $t>0$. 

\end{lem} 
\begin{proof}
Assume that $0\leq F_{1}\leq F_{2}$ and that both of them are bounded and have bounded supports .  Clearly, $F_{2}=F_{1}+H$, where $H$ is not negative. By (\ref{KERN}) we have 
$$
K^{F_{2}}_{t}(x,x)=K^{F_{1}}_{t}(x,x)+K^{H}_{t}(x,x)
$$
where each term is non-negative. The lemma is proven.

\end{proof}

We are going to make use of the heat kernel 
$$
p_{t}(x,y)=\sum_{l=0} ^{\infty}e^{-t \lambda_{l}}u_{l}(x)\overline{u_{l}(y)},
$$ 
which is associated with the heat semigroup $e^{-t\Delta}$ generated by the self-adjoint operator $\Delta$:
$$
e^{-t\Delta}f(x)=\int_{{\bf M}}p_{t}(x,y)f(y)dy.
$$
Note, that in notations (\ref{func-calc-2}), (\ref{KERN}) 
$$
p_{t}(x,y)=K_{t}^{F}(x,y),\>\>\>F(\lambda)=e^{-\lambda^{2}},\>\>\>e^{-t\Delta}=F(t\sqrt{\Delta}).
$$
It is well known that in the case of a compact Riemannian manifold this kernel obeys the following short-time Gaussian estimates:
\begin{equation}\label{heatkern}
 C_{1} t^{-d/2}e^{-c_{1}\frac{ \left(dist(x,y)\right)^{2}}{t}}\leq p_{t}(x,y)\leq 
 C_{2} t^{-d/2}e^{-c_{2}\frac{ \left(dist(x,y)\right)^{2}}{t}}
\end{equation}
where $0<t<1$,  $\>d=dim\>{\bf M}$ and every constant depends on ${\bf M}$.

\subsection{The lower estimate}

We now sketch the proof of the opposite estimate by comparing $\left|M_{\omega^{-1/2}}\right|$ to the number of eigenvalues (counted with multiplicities) in the interval $[0,\>\omega]$.

Inequalities (\ref{2balls})  in conjunction with (\ref{heatkern}) 
it gives for $0<t<1$
\begin{equation}\label{dest}
a_{1}C_{1}|B(x, t^{-1/2})|\leq p_{t}(x,x)=\sum_{l=0}^{\infty}e^{-t\lambda_{l}^{2}}|u_{l}(x)|^{2}\leq a_{2}C_{2}|B(x, t^{-1/2})|.
\end{equation}

\begin{lem}\label{key}
There exist   constants $ A_{1}=  A_{1}({\bf M}) >0,\>\>\> A_{2}=  A_{2}({\bf M}) >0$ such that for all sufficiently large $s>0$ 

\begin{equation}\label{double}
\frac{A_{1}}{|B(x, s^{-1})|}\leq \sum_{l,\>\lambda_{l}\leq s}|u_{l}(x)|^{2}\leq \frac{A_{2}}{|B(x, s^{-1})|}.
\end{equation}

\end{lem}
\begin{proof}

First, we note that using the right-hand side of  (\ref{dest}), Lemma \ref{order} and the inequality
$$
\chi_{[0,\>s]}(\lambda)\leq ee^{-s^{-2}\lambda^{2}}
$$ 
we obtain
\begin{equation}\label{kern-2}
\sum_{l,\>\lambda_{l}\leq s}|u_{l}(x)|^{2}\leq   e\sum_{l,\>\lambda_{l}\leq s}  e^{-s^{-2}\lambda_{l}^{2}}|u_{l}(x)|^{2}\leq 
$$
$$
e\sum_{l\in \mathbb{N}}  e^{-s^{-2}\lambda_{l}^{2}}|u_{l}(x)|^{2}=ep_{s^{-2}}(x,x)\leq \frac{A_{2}}{|B(x, s^{-1})|},\>\>\>\>A_{2}=A_{2}({\bf M}).
\end{equation}
To prove the left-had side of   (\ref{double}) consider the inequality 
$$
e^{-t\lambda^{2}}=e^{-t\lambda^{2}}\chi_{[0, \>s]}+\sum_{j\geq 0}\chi_{[2^{j}s,\>2^{j+1}s]}(\lambda)e^{-t\lambda^{2}}\leq 
$$
$$
\chi_{[0, \>s]}+\sum_{j\geq 0}\chi_{[0,\>2^{j+1}s]}(\lambda)e^{ -t2^{2j}s^{2}  },
$$
which implies 
\begin{equation}\label{b-h-k}
 p_{t}(x,x)\leq K^{\chi_{[0,  s]}}_{1}(x,x)+\sum_{j>0}e^{ -t2^{2j}s^{2}  }K^{j}_{1}(x,x),
\end{equation}
where  
$$
K^{\chi_{[0,  s]}}_{1}(x, x)=\sum_{l,\>\lambda_{l}\leq s}|u_{l}(x)|^{2},
$$
 $K^{\chi_{[0,  s]}}_{1}(x, y)$ being the kernel of the operator $\chi_{[0,  s]}\left(\sqrt{\Delta}\right)$ and 
 $$
 K^{j}_{1}(x,x)=\sum_{l,\>\lambda_{l}\leq 2^{j+1}s}|u_{l}(x)|^{2},
 $$
  $ K^{j}_{1}(x, y)$       being the kernel of the operator $\chi_{[0,\>2^{j+1}s]}(\sqrt{\Delta})$.
  In conjunction   with (\ref{heatkern}) it gives 
\begin{equation}\label{b-h-k-10}
c_{3}|B(x, t^{-1/2})|\leq p_{t}(x,x)\leq K^{\chi_{[0,  s]}}_{1}(x,x)+\sum_{j>0}e^{ -t2^{2j}s^{2}  }K^{j}_{1}(x,x)=
$$
$$
\sum_{l,\>\lambda_{l}\leq s}|u_{l}(x)|^{2}+\sum_{j>0}e^{ -2^{2j}ts^{2}  }\sum_{l,\>\lambda_{l}\leq 2^{j+1}s}|u_{l}(x)|^{2}.
\end{equation}
 Note, that according to (\ref{doubling})  if $\rho>1$ and $\rho s^{-1}$ is sufficiently small then 
\begin{equation}\label{D-cond}
|B(x, \rho s^{-1}|\leq c\rho^{d}|B(x, s^{-1})|,\>\>\> d=dim\>{\bf M}.
\end{equation}
Next, given $s\geq 1$ and $m\in \mathbb{N}$ we pick $t$ such that 
\begin{equation}\label{cond}
s\sqrt{t}=2^{m}.
\end{equation}
The  inequality (\ref{D-cond})  and the condition (\ref{cond}) imply 
\begin{equation}\label{D1}
\frac{(c2^{md})^{-1}}{|B(x, s^{-1})|}\leq \frac{1}{|B(x, 2^{m}s^{-1})|}\leq c_{1}|B(x, t^{-1/2})|, \>\>\>m\in \mathbb{N}, 
\end{equation}
and
\begin{equation}\label{D2}
\frac{1}{|B(x, 2^{-m-1}s^{-1})|}\leq \frac{c2^{(m+1)d}}{|B(x, s^{-1})|},\>\>\>d=dim\>{\bf M}.
\end{equation}
Thus according to (\ref{D1}),           (\ref{kern-2}) and (\ref{b-h-k-10})                       we obtain that for a certain constant $c_{2}=c_{2}({\bf M})$ 
\begin{equation}\label{Eq}
\frac{(c2^{md})^{-1}}  {|B(x, s^{-1})|}\leq c_{1}|B(x, t^{-1/2})|\leq
$$
$$
c_{2}\left( \sum_{l,\>\lambda_{l}\leq s}|u_{l}(x)|^{2}+  \sum_{j>0}e^{ -2^{2j}ts^{2}  }\sum_{l,\>\lambda_{l}\leq 2^{j+1}s}|u_{l}(x)|^{2}\right)\leq
$$
$$
c_{2}\left( \sum_{l,\>\lambda_{l}\leq s}|u_{l}(x)|^{2}+  \sum_{j>0}\frac{ e^{ -2^{2j}ts^{2}  }}{|B(x, 2^{-j-1}s^{-1})|}\right).
\end{equation}
Using  (\ref{Eq}), (\ref{cond}),  and (\ref{D2}) we obtain that for a certain constant $a=a({\bf M})$ 
$$
\frac{(c2^{md})^{-1}}  {|B(x, s^{-1})|}\leq
a\left( \sum_{l,\>\lambda_{l}\leq s}|u_{l}(x)|^{2}+    \sum_{j>0}\frac {  e^{ -2^{2j}2^{2m } } 2^{(j+1)d} }{|B(x, s^{-1})|}\right)\leq
$$
$$
c_{2}\left( \sum_{l,\>\lambda_{l}\leq s}|u_{l}(x)|^{2}+   \frac{2^{d}}  {|B(x, s^{-1})|} \sum_{j>0}  e^{ -2^{2j}2^{2m } } 2^{jd} \right).
$$
Since
$$
 \frac{2^{d}}  {|B(x, s^{-1})|} \sum_{j>0}  e^{ -2^{2j}2^{2m } } 2^{jd}\leq  \frac{2^{d}2^{-md}}  {|B(x, s^{-1})|} \sum_{j>0}  e^{ -2^{2j}2^{2m } } 2^{(j+m)d}\leq 
 $$
 $$
  \frac{2^{d}2^{-md}}  {|B(x, s^{-1})|} \sum_{j>0}  e^{ -2^{2(j+m) } } 2^{(j+m)d}\leq   \frac{2^{d}2^{-md}}  {|B(x, s^{-1})|} \sum_{j>m}  e^{ -2^{2j } } 2^{jd},
 $$
one has that there are positive constants $c_{3},\>c_{4}$ such that for all  sufficiently large $s$ and $m\in \mathbb{N}$
$$
 \frac{2^{-md}}  {|B(x, s^{-1})|} \left( c_{3}-c_{4}2^{d}\sum_{j\geq m}  e^{ -2^{2j}2^{j d} } \right)\leq  \sum_{l,\>\lambda_{l}\leq s}|u_{l}(x)|^{2},
$$
where expression in parentheses is positive for sufficiently large $m\in \mathbb{N}$.
It proves  the left-had side of   (\ref{double}).

\end{proof}
We apply this lemma when  $t =\omega$ to  obtain the following inequality for sufficiently large  $\omega$: 
\begin{equation}\label{ball}
\frac{1}{|B(x, \omega^{-1/2})|}\leq c_{5}p_{\omega}(x,x)\leq c_{6}\sum_{l,\>\lambda_{l}\leq \omega}|u_{l}(x)|^{2}.
\end{equation}
 One has
 $$
|M_{\omega^{-1/2}}|=\sum_{x_{j}\in M_{\omega^{-1/2}}}\frac{|B(x_{j}, \omega^{-1/2})|}{|B(x_{j}, \omega^{-1/2})|}=\sum_{x_{j}\in M_{\omega^{-1/2}}}\frac{1}{|B(x_{j}, \omega^{-1/2})|}
\int_{B(x_{j}, \omega^{-1/2})}dx,
$$
and thanks to  (\ref{2balls})  we also have
$$
\frac{1}{|B(x_{j}, \omega^{-1/2})|}
\int_{B(x_{j}, \omega^{-1/2})}dx\leq \frac{a_{2}}{a_{1}}\int_{B(x_{j}, \omega^{-1/2})}\frac{dx}{|B(x, \omega^{-1/2})|}.
$$
Now, the inequality (\ref{ball}) shows  that for every sufficiently large $\omega>0$ and every $\omega^{-1/2}$-lattice $M_{\omega^{-1/2}}$ the following inequalities hold true
$$
|M_{\omega^{-1/2}}|\leq \frac{a_{2}}{a_{1}}\sum_{x_{j}\in M_{\omega^{-1/2}}}\int_{B(x_{j}, \omega^{-1/2})}\frac{dx}{|B(x, \omega^{-1/2})|}\leq 
$$
$$
\frac{a_{2}}{a_{1}}\int_{{\bf M}}\frac{dx}{|B(x, \omega^{-1/2})|}\leq \frac{a_{2}}{a_{1}}c_{6}\int_{{\bf M}}\left(\sum_{l,\>\lambda_{l}\leq\omega}|u_{l}(x)|^{2}\right)dx.
$$
Since
\begin{equation}\label{N}
\int_{{\bf M}}\left(\sum_{l,\>\lambda_{l}\leq \omega}|u_{l}(x)|^{2}\right)dx=\sum_{l,\>\lambda_{l}\leq  \omega}\int_{{\bf M}}|u_{l}(x)|^{2}dx=\mathcal{N}_{ \omega}.
\end{equation}
we receive  the inequality 
$$
\left|M_{\omega^{-1/2}}\right|\leq c_{7}\mathcal{N}_{ \omega},
$$
for a certain $c_{7}=c_{7}({\bf M})>0$. 
Thus
there exists an $a=a({\bf M})>0$ such that 
\begin{equation}\label{lower}
a\>\sup|M_{\omega^{-1/2}}|\leq \mathcal{N}_{\omega},
\end{equation}
where $\left|M_{\omega^{-1/2}}\right|$ is the number of points in a lattice $M_{\omega^{-1/2}}$ and $\sup$ is taken over all such lattices. The inequalities (\ref{upper}) and (\ref{lower}) prove Theorem \ref{WWL}.

\end{document}